\documentclass{amsart}
\usepackage{amsmath}%
\usepackage{amsfonts}%
\usepackage{amssymb}%
\usepackage{graphicx}
\usepackage{mathtools, tikz, float,hyperref}

\usepackage{color}

\newcommand{\A}{\mathcal{A}}

\newcommand{\B}{\mathcal{B}}
\newcommand{\C}{\mathcal{C}}

\newcommand{\PC}{\mathbb{P}\mathcal{C}}

\newcommand{\G}{\mathcal{G}}
\newcommand{\h}{\mathbb{H}^2}
\newcommand{\M}{\mathcal{M}}
\newcommand{\N}{\mathbb{N}}

\newcommand{\p}{\mathcal{P}}

\newcommand{\R}{\mathbb{R}}

\newcommand{\T}{\mathcal{T}}

\newcommand{\Z}{\mathbb{Z}}

\newcommand{\ML}{\mathcal{ML}}

\DeclareMathOperator{\sys}{sys}
\DeclareMathOperator{\Col}{Col}

\newtheorem{theorem}{Theorem}[section]
\theoremstyle{plain}
\newtheorem{lem}[theorem]{Lemma}

\newtheorem{cor}[theorem]{Corollary}
\newtheorem{rem}[theorem]{Remark}

\newtheorem{conj}[theorem]{Conjecture}
\newtheorem{ques}[theorem]{Question}

\theoremstyle{definition}

\newtheorem*{ProjThm}{Theorem \ref{thm:No_projection}}

\newcounter{jenyacomments}

\title{An extension of the Thurston metric to projective filling currents}
\author{Jenya Sapir}

\begin{document}
\begin{abstract}
We study the geometry of the space of projectivized filling geodesic currents $\PC_{fill}(S)$.  Bonahon showed that Teichm\"uller space, $\T(S)$ embeds into $\PC_{fill}(S)$. We extend the symmetrized Thurston metric from $\T(S)$ to the entire (projectivized) space of filling currents, and we show that $\T(S)$ is isometrically embedded into the bigger space. Moreover, we show that there is no quasi-isometric projection back down to $\T(S)$. Lastly, we study the geometry of a length-minimizing projection from $\PC_{fill}(S)$ to $\T(S)$ defined previously by Hensel and the author.
\end{abstract}

\maketitle

\section{Introduction}
The Thurston metric was first proposed by Thurston for the Teichm\"uller space, $\T(S)$, of a closed, genus $g \geq 2$ surface in \cite{Thurston88}. The metric was defined in terms of optimal Lipschitz stretch factors between surfaces, but in the same paper, he shows that it can, in fact, be expressed in terms of lengths of curves where for any two $X, Y \in \T(S)$, we set
\[
 d(X,Y) = \sup_\alpha \log \frac{\ell_X(\alpha)}{\ell_Y(\alpha)}
\]
where the supremum is taken over all simple closed curves on $S$.

This is a complete, geodesic metric on $\T(S)$, but it is asymmetric. Various ways of symmetrizing the Thurston metric have been proposed, and in fact one of them, called the Length Spectrum Metric, was studied several decades prior by Sorvali in \cite{Sorvali72}. In this paper, we consider a slightly different symmetrization, and let
\[
 d_{Th}(X,Y) = d(X,Y) + d(Y,X)
\]
be the symmetrized Thurston metric on $\T(S)$.

By work of Bonahon \cite{Bon88}, $\T(S)$ embeds into the space of filling geodesic currents, $\C_{fill}(S)$. This is an infinite-dimensional space that contains embeddings of many other spaces of metrics on $S$, as well as all filling closed curves. The usual geometric intersection function $i(\cdot, \cdot)$ extends to a continuous, symmetric, bilinear intersection  on $\C_{fill}(S)$. With this, the embedding of $\T(S)$ into $\C_{fill}(S)$ is natural in the following way. If $X \in \C_{fill}(S)$ represents a point in $\T(S)$, and $\alpha$ represents any closed curve, then $i(X, \alpha)$ is the length of the geodesic representative of $\alpha$ in $X$. Moreover, there is a natural $\R^+$ action on $\C_{fill}(S)$. Quotienting by this action gives us the set of projectivized filling currents, $\PC_{fill}(S)$. 

It turns out that we can extend the symmetrized Thurston metric to all of $\PC_{fill}(S)$.
\begin{theorem}
\label{thm:ProjThm}
 For any $[\mu], [\nu] \in \PC_{fill}(S)$, choose representatives $\mu, \nu \in \C_{fill}(S)$. Then the function 
 \[
  d_{Th}([\mu], [\nu]) = \sup_{\lambda \in \C(S)} \log \frac{i(\nu, \lambda)}{i(\mu, \lambda)} + \sup_{\lambda \in \C(S)} \log \frac{i(\mu, \lambda)}{i(\nu, \lambda)}
 \]
 gives a complete, proper, mapping class group-invariant metric on $\PC_{fill}(S)$. Moreover, when restricted to $\T(S)$, it is exactly the symmetrized Thurston metric.
\end{theorem}

In the process of the proof, we will note several choices we made that are necessary to make sure the metric is well-defined and non-degenerate. First, the two summands of $d_{Th}$ are not actually functions on $\PC_{fill}(S)$. However, it turns out that the sum is, indeed, well-defined (Lemma \ref{lem:Function_on_Projective_Currents}). Next, it is necessary to take the supremum over all currents, not just over measured laminations as is done in the Thurston metric, or else the metric would be degenerate (Remark \ref{rem:Supremum_over_all_currents}). Lastly, we can scale each summand as in Equations \ref{eq:S} and \ref{eq:d_S} to make each one a function on $\PC_{fill}(S)$. Even then, however, each summand individually is only a pseudo-metric (as shown in Lemma \ref{lem:Asymetric_Degenerate}). It is an interesting question whether a different scaling would lead to a non-degenerate asymmetric metric:
\begin{ques}
 Is it possible to define a metric on $\PC_{fill}(S)$ that extends the asymmetric Thurston metric on $\T(S)$?
\end{ques}

\subsection{Connection to hyperbolic, left invariant pseudo-metrics}
By \cite{Eduardo,EC}, $\PC_{fill}(S)$ actually isometrically embeds in a much larger space, denoted $D(G)$. If $G$ is a surface group, then $D(G)$ is (a quotient of)   the space of hyperbolic, left invariant pseudo-metrics for $G$ that are quasi-isometric to a word metric. In particular, in \cite{Eduardo}, Oregon-Reyes defines the so-called (symmetrized) Dilatation metric $\Delta$ on $D(G)$. In a subsequent paper, Oregon-Reyes and Cantrell \cite{EC} prove that $\PC_{fill}(S)$ embeds into $D(G)$. Then it follows directly from the definitions of $d_{Th}$ on $\PC_{fill}(S)$ and $\Delta$ on $D(G)$ that this must be an isometric embedding. However, this cannot be an isometry. In fact, $d_{Th}$ is a proper metric. However, in upcoming work of Oregon-Reyes and Cantrell, they show that $\Delta$ is not.

\subsection{Geometry of $\PC_{fill}(S)$ with the symmetrized Thurston metric}
Since $d_{Th}$ restricts to the usual symmetrized Thurston metric on Teichm\"uller space, we see that $\T(S)$ is isometrically embedded in $\PC_{fill}(S)$. We want to show that $\PC_{fill}(S)$ is much bigger than $\T(S)$. In proofs of Theorem \ref{thm:No_projection} and Lemma \ref{lem:Define_mu} we construct two different families of currents that are far from Teichm\"uller space.

We say that a map $\phi: \PC_{fill}(S) \to \T(S)$ is a quasi-isometric projection if $\phi$ is a quasi-isometry that restricts to a quasi-isometry on $\T(S)$ as a subspace of $\PC_{fill}(S)$. We show that no such map exists:
\begin{theorem}
\label{thm:No_projection}
 There is no quasi-isometric projection from $\PC_{fill}(S)$ to $\T(S)$ with respect to the symmetrized Thurston metric. 
\end{theorem}
To prove this theorem, we will show that there are projective filling currents that are arbitrarily far from any point in Teichm\"uller space. We give a simple example of such a family of curves to prove this theorem, and give a much larger family with more properties in Theorem \ref{thm:Diameter_unbounded_precise}.

This result would imply that $\PC_{fill}(S)$ is not quasi-isometric to $\T(S)$, as long as any quasi-isometric embedding of $\T(S)$ into itself is, in fact, a quasi-isometry. This is not yet known for the symmetrized Thurston metric. Still, we make the following conjecture:
\begin{conj}
 $\PC_{fill}(S)$ is not quasi-isometric to Teichm\"uller space with the symmetrized Thurston metric.
\end{conj}

In \cite{HS21}, Hensel and the author showed that there is a length-minimizing projection 
\[
\pi : \PC_{fill}(S) \to \T(S)
\]
that fixes Teichm\"uller space and is continuous, proper, and mapping class group invariant. In particular, for each $X \in \T(S)$, the set $\pi^{-1}(X)$ is compact and so its diameter with respect to $d_{Th}$ is bounded. However, we will show that this diameter goes to infinity as $X$ goes to the boundary. 
\begin{theorem}
\label{thm:Diameter_unbounded}
 For every $D > 0$ there is an  $X \in \T(S)$ for which the diameter of $\pi^{-1}(X)$ is at least $D$. 
\end{theorem}
In fact, we will prove more than this: in Theorem \ref{thm:Diameter_unbounded_precise}, we will show that over any $\epsilon$-thin part of $\T(S)$ (with short curves specified), we can find a current $\mu$ arbitrarily far from all points in $\T(S)$.

Of course, if $\mu$ is arbitrarily far from all points in $\T(S)$, it is also far from its projection $\pi(\mu)$, and so the diameter of its fiber is large. It would be interesting to know if the diameter of this fiber (coarsely) realizes the distance of $\mu$ to $\T(S)$. In particular, we would like to know about the relationship between the length-minimizing projection $\pi$, and the \textit{distance-minimizing} projection for $d_{Th}$. Since $d_{Th}$ is not likely to be uniquely geodesic, we pose this question as follows:
\begin{ques}
 Is the length-minimizing projection $\pi$ close to a nearest-point projection for $d_{Th}$?
\end{ques}

\subsection{Organization of the paper}
Section \ref{sec:Background} provides some background on geodesic currents, and defines some notation used in later sections. We prove Theorem \ref{thm:ProjThm} about the existence and basic properties of the symmetrized Thurston metric in Sections \ref{sec:Well-defined} through \ref{sec:Complete_metric}. In Section \ref{sec:No_projection} we prove Theorem \ref{thm:No_projection}, which shows there is no quasi-isometric projection from filling projective currents onto $\T(S)$. Then we prove Theorem \ref{thm:Diameter_unbounded} about a large family of currents that are arbitrarily far from all points in $\T(S)$, in Sections \ref{sec:When_curves_short} through \ref{sec:Unbounded_diameter}.
 \subsection{Acknowledgement}
 We would like to thank Vincent Delecroix for the discussions at Oberwolfach that inspired us to write this paper. We would also like to thank Didac Martinez-Granado and Eduardo Oregon-Reyes for telling us about the very interesting connection the space of pseudo-metrics $D(G)$.

\section{Background}
\label{sec:Background}
The space of geodesic currents is defined as follows. Fix a hyperbolic metric $X$ on a genus $g$ surface $S$ with negative Euler characteristic. Then we can identify the universal cover of $S$ with the hyperbolic plane $\h$. The space $\G$ of all unoriented, complete geodesics in $\h$ is then in bijective correspondence with pairs of distinct pairs of points $(x,y)$ on the boundary at infinity, up to the identity $(x,y) \sim (y,x)$. In other words, $\G \simeq \partial \h \times \partial \h \setminus \Delta/ \sim$, where $\Delta$ is the diagonal, and $\sim$ is the $\Z_2$ action. Then $\pi_1(S)$ acts on $\h$, and this action extends to an action on $\G$. The space $\C(S)$ of geodesic currents is then defined as the space of all $\pi_1(S)$-invariant Borel measures on $\G$. We then endow $\C(S)$ with the weak$^*$ topology. Note that a priori, $\C(S)$ depends on the choice of metric $X$. However, Bonahon showed that the spaces obtained using any two metrics $X$ and $Y$ are Holder equivalent.

There is a natural action of $\R^+$ on $\C(S)$, given by scaling each measure by a positive constant. Quotienting by this action gives the space of projective geodesic currents, $\PC(S)$. This space is compact, and contains embedded copies of the set of all closed geodesics on $S$, and Teichmuller space \cite{Bon88}.

The space $\C(S)$ is endowed with a continuous, symmetric, bilinear intersection form $i(\cdot, \cdot)$. It has the property that if $\mu, \nu$ represent closed geodesics, then $i(\mu, \nu)$ gives their geometric intersection number. However, if $\mu$ represents a closed curve, and $\nu$ represents a metric, then $i(\mu, \nu)$ is the length of the geodesic representative of $\mu$ with respect to $\nu$.

A geodesic current $\mu$ is \textbf{filling} if and only if it has positive intersection with all currents. Equivalently, any complete geodesic $\gamma \in \G$ intersects some geodesic in the support of $\mu$. In the case where $\mu$ represents a (non-simple) closed geodesic, this is equivalent to the usual definition, where $\mu$ cuts $S$ into simply connected regions. Given a geodesic current $\mu$, we can define its \textbf{systolic length} by 
\[
 \sys(\mu) = \inf_{\alpha} i(\mu, \alpha)
\]
where the infimum is taken over all closed curves. In \cite[Corollary 3.26]{BIPP}, they show that we can, in fact, take the infimum over simple closed curves. This was shown in \cite{BIPP} to be equivalent to taking the infimum over all measured laminations, or, in fact, over all geodesic currents. In \cite{BIPP}, they show that $\sys(\mu) = 0$ if and only if $\mu$ is non-filling. Since the condition $\sys(\mu) = 0$ is invariant under scaling $\mu$, we see that we can define filling and non-filling projective currents, as well. We let $\C_{fill}(S)$ and $\PC_{fill}(S)$ denote the spaces of filling currents and filling projective currents, respectively.

We note the following fact, that is found in \cite[Theorem 4.1]{BIPP} in the case of non-filling currents, and is straightforward for filling currents:
\begin{lem}
The systolic length of each current is realized by a measured lamination. That is, for each $\mu \in \C(S)$, there is a $\lambda \in \ML(S)$ so that $i(\mu, \lambda) = \sys(\mu)$.
\end{lem}
\begin{proof}
For a non-filling current $\mu$, they show in \cite[Theorem 4.1]{BIPP} that there is some measured lamination $\lambda$ so that $i(\mu, \lambda) = 0 = \sys(\mu)$. 

In the case where $\mu$ is a filling current, fix a hyperbolic metric $X$. We note that the function $f$ sending each $\lambda \in \C(S)$ to $\frac{i(\mu, \lambda)}{\ell_X(\lambda)}$ is invariant under scaling $\lambda$, and so is in fact a function $f: \PC(S) \to \R$. Since $\PC(S)$ is compact, this function attains its minimum. Moreover, as $\mu$ is filling, $i(\mu, \lambda) > 0$ for all $\lambda$. Thus, there is some constant $c$ so that 
\[
 c \leq \frac{i(\mu, \lambda)}{\ell_X(\lambda)}
\]
for each $\lambda \in \C(S)$. In particular, suppose $\sys(\mu) = d$. Then for all simple closed curve $\alpha$ so that $i(\mu, \alpha) < 2d$, we have $\ell_X(\alpha) < 2d/c$. Since $\sys(\mu) = \inf_\alpha i(\mu, \alpha)$, we can just take the infimum over those simple closed curves $\alpha$ with length at most $2d/c$ with respect to $X$. But this is a finite set, and so the infimum is realized at some simple closed curve.
\end{proof}

\subsection{Notation}
Given quantities $f, g$, we use say $f\prec g$ with constants depending only on some quantity $C$ if there is a constant $c$ depending only on $C$ so that $f \leq c g$. Likewise, we say $f \succ g$ if $f \geq c' g$ for some $c'>0$ depending only on $C$, and $f \asymp g$ if $f \prec g$ and $g \prec f$.

\section{Well-defined on $\PC_{fill}(S)$}
\label{sec:Well-defined}
If we fix $\lambda \in \C(S)$, we can define the function sending a pair $\mu, \nu \in \C(S)$ to $\frac{i(\nu, \lambda)}{i(\mu, \lambda)}$. This is defined only on pairs of currents $\mu$ and $\nu$, and does not descend to pairs of projective currents. Still, we show the following:
\begin{lem}
\label{lem:Function_on_Projective_Currents}
 The function $d_{Th}$ is in fact well-defined as a function from pairs of filling projective currents to $\R$.
\end{lem}
\begin{proof}
Let $\mu, \nu \in \C_{fill}(S)$, and let $t,s \in \R_{>0}$. 
Then a straightforward computation gives $d_{Th}(t\mu, s\nu) = d_{Th}(\mu, \nu)$. Thus, $d_{Th}$ descends to a function on projective currents.

We need to check that $d_{Th}(\mu, \nu) < \infty$ for all pairs of filling currents. (This is not true if $\mu$ or $\nu$ is not filling, as shown in Lemma \ref{lem:complete_metric} below.) Consider the function $f: \PC(S) \to \R$ given by 
\[
 f(\lambda) = \frac{i(\mu, \lambda)}{i(\nu,\lambda)}
\]
Since $\mu, \nu$ are filling, neither numerator nor denominator is ever 0. Then $f$ is a continuous function on a compact set, and so it achieves a (non-zero) maximum. In particular, $d_{Th}(\mu, \nu) < \infty$ for all pairs of filling currents $\mu, \nu$.
\end{proof}

Thus, $d_{Th}(\mu, \nu)$ is a function on pairs of projective currents, even though the two summands are not. However, we can rewrite $d_{Th}$ as the sum of two functions on projective currents. Given two currents $\mu, \nu \in \C(S)$, we can define the $S : \C(S) \times \C(S) \to \R$ by
\begin{equation}
\label{eq:S}
 S(\mu, \nu) = \frac{\sys(\mu)}{\sys(\nu)} \sup_{\lambda \in \C(S)} \frac{i(\nu, \lambda)}{i(\mu, \lambda)}
\end{equation}
In fact, $S$ is scale-invariant, and so depends only on the class $[\mu], [\nu] \in \PC_{fill}(S)$. 

From $S$, we can define the function
\[
 d_S : \PC_{fill}(S) \times \PC_{fill}(S) \to \R
\]
by
\begin{equation}
\label{eq:d_S}
 d_S(\mu, \nu) = \log S(\mu, \nu)
\end{equation}
Unfortunately, it turns out that $d_S$ is not only asymmetric, but also degenerate, as we show in Lemma \ref{lem:Asymetric_Degenerate}. However, when we symmetrize it, we get back our function $d_{Th}$. In fact, 
\begin{align*}
 d_S(\mu, \nu) + d_S(\nu, \mu)
 & = \log S(\mu, \nu) + \log S(\nu, \mu) \\
  & = \log \big( S(\mu, \nu) S(\nu, \mu)\big) \\ 
  & = d_{Th}(\mu, \nu)
\end{align*}
because a term of the form $\frac{\sys(\mu)}{\sys(\nu)} \cdot \frac{\sys(\nu)}{\sys(\mu)}$ goes away.  This gives another proof that $d_{Th}$ is well-defined as a function from $\PC_{fill}(S) \times \PC_{fill}(S)$ to $\R$, although this fact can be shown directly.

The function $d_S$ defined in Equations \ref{eq:S} and \ref{eq:d_S} will be useful for the remainder of the results, even if it is not, in itself, a metric on $\PC_{fill}(S)$.

\section{Non-degeneracy}
\begin{lem}
 The function $d_{Th}$ is non-degenerate. That is, for all $[\mu], [\nu] \in \PC_{fill}(S)$, $d_{Th}(\mu, \nu) \geq 0$, and it is equal to 0 if and only if $\mu = \nu$. 
\end{lem}
\begin{proof}
 For what follows, it is useful to view $d_{Th}(\mu, \nu)$ as $d_S(\mu, \nu) + d_S(\nu, \mu)$, where $d_S$ is defined in Equations (\ref{eq:S}) and (\ref{eq:d_S}). We will first show that $d_S(\mu, \nu) \geq 0$ for all $\mu, \nu \in \C(S)$. In fact, let $\lambda \in \ML(S)$ be a measured lamination that realizes the systole of $\mu$, so $\sys(\mu) = i(\mu, \lambda)$. Then
 \[
  \frac{\sys(\mu)}{\sys(\nu)} \frac{i(\nu, \lambda)}{i(\mu, \lambda)}  =  \frac{i(\nu, \lambda)}{\sys(\nu)}
 \]
Since $i(\nu, \lambda) \geq \sys(\nu)$ for all currents $\lambda$, this quantity is bounded below by 1. Thus, $d_S(\mu, \nu) \geq 0$. (Note that it is not true that the two summands used to define $d_{Th}(\mu, \nu)$ are always non-negative.)

Next, it is clear that if $\mu = \nu$, then $d_{Th}(\mu, \nu) = 0$. So suppose that we have $\mu, \nu \in \C(S)$ for which $d_{Th}(\mu, \nu) = 0$. Since $d_S$ is non-negative, this means that $d_S(\mu, \nu) = d_S(\nu, \mu) = 0$. So $d_S(\mu, \nu) = 0$ implies that
\[
 S(\mu, \nu) = \frac{\sys(\mu)}{\sys(\nu)} \sup_{\lambda \in \C(S)} \frac{i(\nu, \lambda)}{i(\mu, \lambda)} = 1
\]
In particular, $i(\nu, \lambda) \sys(\nu) \leq i(\mu, \lambda) \sys(\mu)$ for all $\lambda$, while on the other hand, $d_S(\nu,\mu) = 0$ implies the inequality in the other direction. Thus,
\[
 i(\nu, \lambda) \sys(\nu) = i(\mu, \lambda) \sys(\mu)
\]
for all $\lambda \in \C(S)$. In particular, this equality is true for all closed curves $\lambda$. Since $\mu, \nu$ are filling currents, they have positive systole \cite{BIPP}. So we can scale them so that their systole is 1, meaning that their marked length spectra are the same. But by work of Otal \cite{Otal90}, this means that $\mu = \nu$ as projective currents.
 
\end{proof}

The above proof relied on the fact that if $d_{Th}(\mu, \nu) = 0$, then the marked length spectra of $\mu$ and $\nu$ were scalar multiples of each other.  However, if $d_S(\mu, \nu) = 0$, then that just means that the marked length spectrum of one is strictly smaller than the marked length spectrum of some multiple of the other. This observation allows us to show that $d_S$ is degenerate.

\begin{lem}
\label{lem:Asymetric_Degenerate}
 The asymmetric function $d_S$ is degenerate. That is, we can find examples of $\mu \neq \nu$ for which $d_S(\mu, \nu) = 0$.
\end{lem}
\begin{proof}
 Let $n \geq 3$. Let $\{\alpha_1, \dots, \alpha_n\}$ be a collection of simple closed curves so that $\alpha_1 + \dots + \alpha_n$ is a filling current (Figure \ref{fig:Degenerate}). We can arrange things so that $\alpha_1$ intersects $\alpha_2$ exactly once, and does not intersect any other curve in the list. Then let 
 \[
  \mu = \alpha_1 + \dots + \alpha_n, \text{ and } \nu = \alpha_1 + \dots + \alpha_{n-1} + 2 \alpha_n
 \]
 \begin{figure}[h!]
  \centering 
  \includegraphics{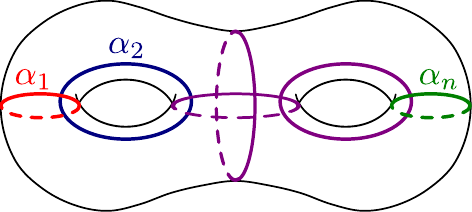}
   \caption{The only difference between $\mu$ and $\nu$ is that in $\mu$, $\alpha_n$ has weight 1, and in $\nu$ it has weight 2.}
  \label{fig:Degenerate}
 \end{figure}

Since $i(\alpha_1, \alpha_2) = i(\alpha_1, \mu) = i(\alpha_1, \nu) =1$, and $i(\beta, \mu), i(\beta, \nu) \in \N$ for all simple closed curves $\beta$, we have that $\sys(\mu) = \sys(\nu) = 1$. Moreover, for all currents $\lambda$, 
\begin{align*}
 i(\lambda, \mu) & = \sum_1^n i(\lambda, \alpha_i) \\
  & \leq 2 i(\lambda, \alpha_n) + \sum_1^{n-1} i(\lambda, \alpha_i) \\
  & = i(\lambda, \nu)
\end{align*}
Since $i(\lambda, \mu) \leq i(\lambda, \mu)$ for all $\lambda \in \C(S)$, and $i(\alpha_1, \mu) = i(\alpha_1, \nu)$, we have
\[
 d_S(\mu, \nu) = 0
\]
for $\mu \neq \nu$.
\end{proof}

\section{Triangle inequality}
We now complete the proof that $d_{Th}(\mu, \nu)$ is a metric.
\begin{lem}
 The function $d_{Th}$ satisfies the triangle inequality.
\end{lem}
\begin{proof}
 Let $\mu, \nu, \eta$ be three filling currents. We have that $d_{Th}(\mu, \eta) = d_S(\mu, \eta) + d_S(\eta, \mu)$. So 
 \[
 d_{Th}(\mu, \nu) + d_{Th}(\nu, \eta) = d_S(\mu, \nu) + d_S(\nu, \eta) + d_S(\nu, \mu) + d_S(\eta, \nu)
 \]
 Since $d_S = \log S$, we must, therefore, show that 
 \[
  S(\mu, \eta)S(\nu, \eta) \leq S(\mu, \nu)S(\nu, \eta)S(\nu, \mu)S(\eta, \nu)
 \]
 We will do this in two parts. First, 
 \begin{align*}
  S(\mu, \nu)S(\nu, \eta) & =  \frac{\sys(\mu)}{\sys(\nu)} \sup_{\lambda \in \C(S)} \frac{i(\nu, \lambda)}{i(\mu, \lambda)} \cdot  \frac{\sys(\nu)}{\sys(\eta)} \sup_{\lambda \in \C(S)} \frac{i(\eta, \lambda)}{i(\nu, \lambda)} \\
  & \geq \frac{\sys(\mu)}{\sys(\eta)} \sup_{\lambda \in \C(S)} \frac{i(\eta, \lambda)}{i(\mu, \lambda)}\\
  & = S(\mu, \eta)
 \end{align*}
where we use the fact that the $\sys(\nu)$ terms cancel, and that the product of suprema if greater than the supremum of the product. Likewise, the same arguments give us 
\[
S(\nu, \mu) S(\eta, \nu) \geq S(\eta, \mu)
\]
And so, we are done.
\end{proof}
Note that this proves that $d_S$ also satisfies the triangle inequality.

\section{Restriction to Thurston metric on $\T(S)$}
\begin{lem}
 Our metric is exactly the symmetrized Thurston metric when restricted to $\T(S) \subset \PC_{fill}(S)$.
\end{lem}
\begin{proof}
For all $X, Y \in \T(S)$, we symmetrized the Thurston metric on $\T(S)$ by
\[
 d_{Th}(X,Y) = \sup_{\lambda \in \ML(S)} \log \frac{i(X, \lambda)}{i(Y, \lambda)} + \sup_{\lambda \in \ML(S)} \log \frac{i(Y, \lambda)}{i(X, \lambda)} 
\]
All we have to show is that taking the supremum over measured laminations is the same as taking the supremum over all currents. Since weighted closed geodesics are dense in the space of currents, and all functions involved are continuous, we first note that instead of taking a supremum over all currents, we can simply take a supremum over all closed geodesics. 

Thurston showed that $\sup_{\lambda \in \ML(S)} \log \frac{i(X, \lambda)}{i(Y, \lambda)}$ is, in fact, the same as $\inf \log L$, where the infimum is taken over all Lipschitz constants $L$ of maps $f: X \to Y$ \cite{Thurston88}. Moreover, he showed that the infimum is realized by some map. So suppose $f: X \to Y$ is a map with minimal Lipschitz constant $L$. Then, let $\gamma$ be any closed geodesic on $X$. Then the length of $f(\gamma)$ with respect to $Y$ is at most $L \ell_X(\gamma)$. In particular, $\ell_Y(\gamma) \leq L \ell_X(\gamma)$ for all closed geodesics $\gamma$. So, in fact, 
\[
 \sup_{\gamma - \text{ closed}} \frac{i(Y, \gamma)}{i(X,\gamma)} \leq L =  \sup_{\gamma - \text{ simple}} \frac{i(Y, \gamma)}{i(X,\gamma)} 
\]
and therefore, taking the supremum over all closed curves is the same as taking the supremum over all simple closed curves in the case of $\T(S)$.

It follows immediately that $d_{Th}$ is exactly the symmetrized Thurston metric.
\end{proof}
\begin{rem}
\label{rem:Supremum_over_all_currents}
For generic currents $\mu$ and $\nu$, we need to define $d_{Th}(\mu, \nu)$ as a supremum over all currents to get a non-degenerate metric, as it is possible for $\mu \neq \nu$ to have the same simple marked length spectrum, but impossible for them to have the same (non-simple) marked length spectrum by \cite{Otal90}.
\end{rem}
For example, we can choose a hyperbolic metric $X$ on $S$. Then by work of Birman and Series, the set of all simple geodesics is a nowhere dense subset of $X$ \cite{BS85}, and in fact, by \cite{BP07}, it misses some open ball whose radius depends only on the genus $g$. We can then modify the metric in such a ball slightly to get a new negatively curved metric $X'$ where the simple closed geodesics have the same length as in $X$. Since the set of negatively curved metrics embeds in the space of currents by \cite{Otal90}, this gives us two currents, coming from $X$ and $X'$, with the same simple marked length spectrum.

\section{Complete, proper metric}
\label{sec:Complete_metric}
We know that the (symmetrized) Thurston metric is complete on Teichmuller space, so $d_{Th}$ must be complete when restricted to $\T(S)$. We show that $d_{Th}$ is, in fact, complete on all of $\PC_{fill}(S)$.
\begin{lem}
\label{lem:complete_metric}
 The metric $d_{Th}$ is complete.
\end{lem}
\begin{proof}
 Since $d_{Th}: \PC_{fill}(S) \times \PC_{fill}(S) \to \R$ is continuous as a function, we just need to show that no sequence $\mu_n$ converging to a non-filling current $\mu$ is Cauchy with respect to $d_{Th}$. So suppose $\mu_n$ is a sequence in $\PC_{fill}(S)$ converging to some non-filling current $\mu$. It is enough to show that for every $[\lambda_0] \in \PC_{fill}(S)$, $d_{Th}(\lambda_0, \mu_n) \to \infty$.
 
 To see this, fix a filling current $\lambda_0$. Since $\lambda_0$ is filling, $i(\mu_n, \lambda_0) \geq \sys(\lambda_0)$ for each $n$. So we can bound $S(\nu, \mu_n)$ from below as follows:
 \begin{align*}
  S(\nu, \mu_n) & \geq \frac{\sys(\nu)}{\sys(\mu_n)} \frac{i(\mu_n, \lambda_0)}{i(\nu, \lambda_0)} \\
    & \geq \frac{\sys(\nu)}{\sys(\mu_n)} \frac{\sys(\lambda_0)}{i(\nu, \lambda_0)}
 \end{align*}
Since $\mu$ is not filling, $\sys(\mu) = 0$. The systole function is continuous on $\C(S)$, so $\sys(\mu_n) \to 0$ as $n \to \infty$. The rest of the terms on the right are constant, meaning that the left hand side goes to infinity with $n$. Thus, $d_S(\lambda_0, \mu_n) \to \infty$ as $n \to \infty$. As $d_{Th}(\lambda_0, \mu_n) = \log S(\lambda_0, \mu_n) + \log S(\mu_n, \lambda_0)$, we get that $d_{Th}(\lambda_0, \mu_n)$ diverges.
\end{proof}

It follows from this lemma, and the fact that $\PC(S)$ is compact, that the metric $d_{Th}$ is proper.
\begin{cor}
 The metric $d_{Th}$ is proper.
\end{cor}
\begin{proof}
 We just need to show that for each $L> 0$, the closed ball $\B(\lambda_0, L) = \{\mu \ | \ d_{Th}(\lambda_0, \mu) \leq L\}$ is compact. Since $d_{Th}$ is continuous as a function on pairs of currents, $\B(\lambda_0, L)$ is closed in $\PC_{fill}(S)$. But $\PC(S)$ is compact, so if $\B(\lambda_0, L)$ is not compact, then there is some sequence $\mu_n \in \B(\lambda_0, L)$ that converges to a non-filling current $\mu_\infty$. However, by the proof of the previous lemma, we get that $d_{Th}(\lambda_0, \mu_n)$ goes to infinity as $n \to \infty$. This contradicts the fact that $d_{Th}(\lambda_0, \mu_n) \leq L$ for all $L$. So the metric is proper.
\end{proof}

\section{No quasi-isometric projection}
\label{sec:No_projection}
Let $(X,d_X)$ and $(Y,d_Y)$ be metric spaces. Recall that $\phi: X \to Y$ is a $K,C$-quasi-isometric embedding if for each $x_1, x_2 \in X$, 
\[
 \frac 1K d_Y(\phi(x_1), \phi(x_2)) - C \leq d_X(x_1, x_2) \leq  K d_Y(\phi(x_1), \phi(x_2)) + C
\]
If, moreover, for each $y \in Y$ there is some $x \in X$ so that $d(\phi(x), y) < C$, then $\phi$ is a quasi-isometry. 

If $Y$ is isometrically embedded in $X$, then we say $\phi : X \to Y$ is a quasi-isometric projection if $\phi$ is a quasi-isometry that restricts to a quasi-isometry on $Y$ when viewed as a subset of $X$. Note that, a priori, a quasi-isometry only restricts to a quasi-isometric embedding on proper subsets.

We restate Theorem \ref{thm:No_projection} here:
\begin{ProjThm}
 There is no quasi-isometric projection from $\PC_{fill}(S)$ to $\T(S)$.
\end{ProjThm}
\begin{proof}
 We will first show that, for any $C > 0$, there is a current $\mu$ so that $d_{Th}(\mu, \T(S)) > C$. In fact, decompose $S$ into a union of subsurfaces $Y_1$ and $Y_2$, where both $Y_1$ and $Y_2$ have genus at least 1. Let $\p$ be a pants decomposition of $S$ that includes the boundaries of $Y_1$ and $Y_2$. Extend $\p$ to a marking $\M$ of $S$. That is, for each $\gamma \in \p$, choose a simple closed curve $\bar \gamma$ so that $1 \leq i(\gamma, \bar \gamma) \leq 2$, and for every $\gamma' \neq \gamma \in \p$, $i(\bar \gamma, \gamma') = 0$. 
 
 Let $\M_1 \subset \M$ be the set of curves in $\M$ that pass through the interior of $Y_1$ (Figure \ref{fig:Far_from_Teich}). So, if $\gamma$ is on the boundary of $Y_1$, then $\gamma \not \in \M_1$. But if $\bar \gamma$ is the curve transverse to $\gamma \subset \partial Y_1$, then $\bar \gamma \in \M_1$. Let $\M_2 = \M \setminus \M_1$. Now for each $n > 0$ define $\mu_n \in \C_{fill}(S)$ by 
 \[
  \mu_n = \sum_{\gamma \in \M_1} \gamma + n \sum_{\gamma \in \M_2} \gamma
 \]
So we weigh each curve in $\M_1$ by 1, and each curve in $\M_2$ by $n$.
\begin{figure}[h!]
 \centering 
 \includegraphics{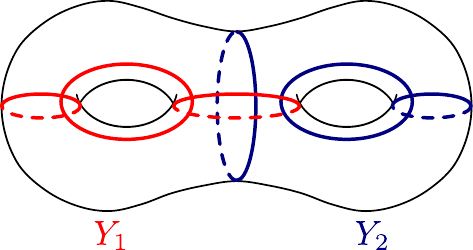}
 \caption{The red curves are in $\M_1$, and the blue curves are in $\M_2$.}
 \label{fig:Far_from_Teich}
\end{figure}

Take any $X \in \T(S)$. We will show that $d_{Th}(\mu_n, X) > n/ d$ where $d$ depends only on the genus $g$ of $S$. In fact, let $Q$ be a Bers pants decomposition of $X$. That is, the length of each curve in $Q$ is bounded by a constant depending only on $g$. As $Y_1$ has genus at least 1, there must be some curve $\beta_1\in Q$ that passes through $Y_1$. Thus, there is some $\gamma_1 \in \M$ contained entirely in $Y_1$ so that $i(\gamma_1, \beta_1) > 0$. Since $\gamma_1 \in \M$ lies entirely in $Y_1$, then it only intersects curves in $\M_1$. Thus, 
\[
 i(\mu_n, \gamma_1) < d
\]
for a constant $d$ depending only on $S$. Since the length of $\beta_1$ in $X$ is bounded from above, and $\beta_1$ intersects $\gamma_1$, the length of $\gamma_1$ must be bounded below. So up to modifying $d$ slightly, 
\[
i(X, \gamma_1) > \frac 1 d
\]
Thus, there is a $\gamma_1 \in \M_1$ so that
\[
 \frac{i(X, \gamma_1)}{i(\mu_n, \gamma_1)} > d^2
\]

Next, let $\beta_2 \in Q$ be a curve that passes through $Y_2$. Then $\beta_2$ intersects some $\gamma_2 \in \M_2$. As $\gamma_2$ has weight $n$ in $\mu_n$, we have 
\[
 i(\mu_n, \beta_2) \geq n
\]
But since $\beta_2$ is part of the Bers pants decomposition of $X$, we have that $i(X, \beta_2) < d$, up to slightly increasing the constant $d$ above. In other words, there is a $\beta_2 \in Q$ so that
\[
  \frac{i(\mu_n, \beta_2)}{i(X, \beta_2)} > \frac{n}{d}
\]
Therefore, 
\[
 d_{Th}(\mu_n, X) > \log n + \log d
\]
for every $X \in \T(S)$.

So suppose $\phi : \PC_{fill}(S) \to \T(S)$ were a quasi-isometric projection from projective filling currents to Teichm\"uller space. As $\phi$ is a quasi-isometry when restricted to $\T(S)$, there is some $C> 0$ so that for each $n$, there is some $X_n \in \T(S)$ for which
\[
 d_{Th}(\phi(\mu_n), \phi(X_n)) < C
\]
But we know that $d_{Th}(\phi(\mu_n), \phi(X_n)) > \log n + \log d$ for a fixed $d$. So by choosing $n$ large enough, we get a contradiction.
\end{proof}
 Using the lemmas in Section \ref{sec:When_curves_short} and the proof of Theorem \ref{thm:Diameter_unbounded_precise}, we can show that the current $\mu$ defined above is also far from its length-minimizing projection, and that in its length-minimizing projection, the boundary curves of $Y_1$ are short, and no other curves are short. However, the above construction only works when we decompose $S$ into surfaces with at least either 1 genus or 4 boundaries, while the construction for Theorem \ref{thm:Diameter_unbounded_precise} allows for more general decompositions.

\section{When curves are short}
\label{sec:When_curves_short}
Next, we turn to the projection $\pi : \PC_{fill}(S) \to \T(S)$ defined in \cite{HS21}. This projection is length minimizing in the sense that, for each $\mu \in \PC_{fill}(S)$, we have 
\[
 i(\mu, \pi(\mu)) < i(\mu, Y)
\]
for all $Y \neq \pi(\mu) \in \T(S)$. 

We cite two theorems from  \cite{Sapir_comparison} that characterize when curves in $\pi(\mu)$ are short, and when subsurfaces of $\pi(\mu)$ are thick, using just the intersection function of $\mu$. These theorems will allow us to come up with many examples of geodesic currents that are far away from their length-minimizing projections to $\T(S)$, and in fact, we can guarantee that these projections are in whichever thin part of $\T(S)$ we like.

The first theorem characterizes the short curves of $\pi(\mu)$:
\begin{theorem}[Theorem 1.2(ii), \cite{Sapir_comparison}]
\label{thm:Short_curves}
 For every $\epsilon > 0$, there is a constant $N$ so that for any $\mu \in PC_{fill}(S)$, any simple closed curve $\alpha$, and any simple closed curve $\beta$ with $i(\alpha, \beta) \geq 1$, if
  \[
   i(\mu, \beta) > N i(\mu, \alpha)
  \]
  then $\ell_{\pi(\mu)}(\alpha) < \epsilon$.

\end{theorem}

Then, this theorem characterizes subsurfaces where no curve is too short:
\begin{theorem}[Theorem 1.3, \cite{Sapir_comparison}]
\label{thm:Thick_part_guarantee}
 Let $Y \subset \pi(\mu)$ be a subsurface so that $\ell_{\pi(\mu)}(\beta) < c_b$ for each boundary component $\beta$ of $Y$, where $c_b$ is the Bers constant. Then for each essential simple closed curve $\alpha$ in $Y$,
\[
 \ell_{\pi(\mu)}(\alpha) \succ 1
\]
if and only if there exists a marking $\Gamma$ of $Y$ so that
 \[
  i(\mu, \gamma) \asymp \sys_Y(\mu)
 \]
for all $\gamma \in \Gamma$, where all constants depend only on $S$.
\end{theorem}

\section{Constructing filling currents}
We will construct a current $\mu \in \C_{fill}(S)$ with special properties. In the next section, we will show that it is far from its length minimizer with respect to the Thurston metric.
\begin{lem}
\label{lem:Define_mu}
 Let $\A = \{\alpha_1, \dots, \alpha_n\}$ be a simple closed multicurve on $S$. For any $N > 0$, there exists a weighted sum of closed curves $\mu \in \C_{fill}(S)$ so that 
 \begin{enumerate}
  \item $i(\mu, \alpha_i) \in \{1,2\}$ for all $\alpha_i \in \A$
  \item If $\beta$ is any closed curve that intersects some $\alpha_i \in \A$, then 
  \[
   i(\mu, \beta) > N i(\mu, \alpha_i)
  \]
  \item There is a marking $\Gamma$ for $S \setminus \bigcup \alpha_i$ so that 
  \[
   i(\mu, \gamma) \asymp 1
  \]
  for each $\gamma \in \Gamma$, for constants depending only on $S$.
  \item For any closed curve $\delta \subset S \setminus \bigcup \alpha_i$ with $i(\delta, \delta) = 1$, we have $i(\mu, \delta) \asymp 1$, where the constant depends on $\delta$, but not on $N$.
 \end{enumerate}
\end{lem}

\begin{proof}
 First, extend $\A$ to a pants decomposition $\p$, and extend this further to a marking $\M$ of $S$. We use $S \setminus \A$ to denote the result of cutting $S$ along the curves in $\A$.

  \begin{figure}[h!]
  \centering
  \includegraphics{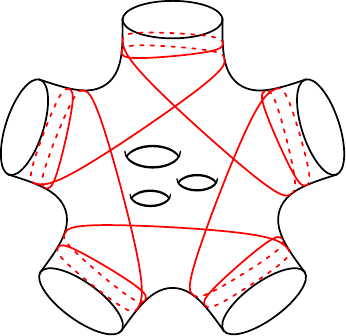}
  \caption{The curve $\eta_Y$ when $N= 1$.}
  \label{fig:Cutoff}
 \end{figure}
 
 Now suppose $Y$ is a connected component of $S \setminus \A$. Up to relabeling, suppose the boundary components of $Y$ are $\alpha_1, \dots, \alpha_m$. Note that a curve in $\A$ can contribute at most 2 boundary components to $Y$ (for example, if $\A$ consists of a single non-separating curve.) Fix a basepoint in $Y$, and choose representatives of $\alpha_1, \dots, \alpha_m$ in $\pi_1(Y)$. Let $\eta_Y$ be the non-simple closed curve represented by 
 \[
  \eta_Y = (\alpha_1)^{N+1}(\alpha_2)^{N+1} \cdots (\alpha_m)^{N+1}
 \]
in $\pi_1(Y)$, where we abuse notation slightly and give the same label to the boundary of $Y$ and its representative in $\pi_1(Y)$ (Figure \ref{fig:Cutoff}). If $Y$ has only 1 boundary curve, $\alpha_1$, then let $\alpha_2$ be any simple closed curve inside $Y$. Let $\eta_Y$ be the curve in the free homotopy class of $\alpha_1^{N+1}\alpha_2$.

 Then we let $\mu$ be the filling current given by
 \[
  \mu = \sum_{\gamma \in \M} \gamma + \sum_{Y \subset S \setminus \A} \eta_Y
 \]
 Since $\mu$ contains a marking of $S$, it is a filling current. Moreover, by definition, $\alpha_1, \dots, \alpha_n$ are all disjoint from $\eta_Y$ for each component $Y$ of $S \setminus \A$. So for each $i$, $\alpha_i$ intersects only its transverse curve in $\M$. Thus, 
 \[
  i(\mu, \alpha_i) \in \{1,2\}
 \]

 Next, suppose $\beta$ is any closed curve that intersects $\alpha_i$. Then $\beta$ crosses $\alpha_i$ and enters some component $Y$ of $S \setminus \A$. Thus, $\beta$ must intersect $\eta_Y$. Moreover, as $\eta_Y$ twists $N+1$ times about $\alpha_i$, we must have that 
 \[
  i(\beta, \eta_Y) \geq  N
 \]
 where we are only guaranteed $N$ intersections, not $N+1$ intersection, because the number of self-intersections of the curve $\eta_Y$  in the cylinder about $\alpha_i$ depends on the orientation of the boundary curves of $Y$ relative to one another.
In particular,
\[
 i(\beta, \mu) \geq N i(\alpha_i, \mu)
\]

Now let $\Gamma$ be the marking $\M$ restricted to $S \setminus \A$. That is, we take all the curves in $\M$ that are essential and non-peripheral in $S \setminus \A$. Let $\gamma \in \Gamma$. Then $\gamma$ lies in some component $Y$ of $S \setminus \A$. So $\gamma$ can only intersect $\eta_Y$, or other curves in $\M$. We have that $i(\M, \M) \asymp 1$ by the definition of a marking, where the constant depends linearly on $\chi(S)$. So we just need to control $i(\gamma, \eta_Y)$. 

Suppose the boundary of $Y$ consists of curves $\alpha_1, \dots, \alpha_m$. Choose a point $x_i$ on $\alpha_i$ for each $i$. Then for each $i= 1, \dots, m$ we can find an arc $\delta_i$ joining $x_i$ to $x_{i+1}$ (were we set $x_{m+1} = x_1$) so that $\eta_Y$ is freely homotopic to a curve that winds $N+1$ times around $\alpha_1$ starting and ending at $x_1$, follows $\delta_1$, winds $N+1$ times around $\alpha_2$ and so on, until it follows $\delta_m$ back from $x_m$ on $\alpha_m$ to $x_1$ on $\alpha_1$. If $Y$ has one boundary curve, then take $\delta_1$ to be an arc from $x_1$ to itself, so that, when we close it up at $x_1$, it is freely homotopic to the simple closed curve $\alpha_2$ in $Y$. 

As $\gamma$ lies inside $Y$, it only intersects $\delta_1, \dots, \delta_m$, and not any of the curves on $\partial Y$. Moreover, we can choose the arcs $\delta_1, \dots, \delta_m$ independent of the number of times $\eta_Y$ winds around the boundary curves. These arcs just depend on the choice of representative of $\alpha_1, \dots, \alpha_m$ in $\pi_1(S)$. Thus, 
\[
 i(\gamma, \eta_Y) \leq \sum_{i=1}^m i(\gamma, \delta_i)
\]
where $i(\gamma, \delta_i)$ is the least number of intersection points between a curve isotopic to $\gamma$ inside $Y$, and a curve isotopic to $\delta_i$ relative its endpoints. As the right-hand side depends only on the topology of $Y$, we have $i(\gamma, \eta_Y) \prec 1$. In other words, 
\[
 i(\gamma, \mu) \prec 1
\]
for constants only depending on $D$.

Lastly, let $\delta$ be any curve with one self-intersection in some complementary region $Y$ of $S \setminus \A$. Then $\delta$ intersects the marking $\M$, and the curve $\eta_Y$. The intersection of $\delta$ with $\M$ depends on $\delta$, but is independent of $N$. This is also true for the intersection of $\delta$ with $\eta_Y$. To see this, consider the curve above homotopic to $\eta_Y$, that is a concatenation of the arcs $\delta_1, \dots, \delta_m$ and the boundary curves of $Y$. Since $\delta$ is disjoint from the boundary curves of $Y$, it only intersects the arcs $\delta_1, \dots, \delta_m$. So $i(\delta, \eta_Y)$ depends on $\delta$, but is independent on $N$. In other words ,$i(\delta, \mu) = O(1)$.
\end{proof}

\section{Far-away currents}
\label{sec:Unbounded_diameter}
For any point $X \in \T(S)$, its pre-image under the length-minimizing projection is compact, and therefore has bounded diameter. However, in any thin part of $\T(S)$, we can find metrics $X$ where $p^{-1}(X)$ has arbitrarily large diameter. More than that, we construct currents $\mu$ lying over any thin part, which are arbitrarily far away from \textit{any} point in $\T(S)$.


To make this precise, let $\epsilon, c > 0$. Let $\A$ be a simple closed multi-curve on $S$. Then define the $\A, \epsilon$-thin part of Teichm\"uller space, $\T_{\A, \epsilon}(S)$, to be the set  of those points $X \in \T(S)$ where $\ell_X(\alpha) < \epsilon$ for all $\alpha \in \A$, and $\ell_X(\beta) > c$ for all $\beta \not \in \A$. So this is a subset of the $\epsilon$-thin part of $\T(S)$ where the short curves are exactly those in $\A$. We show that we can find currents $\mu$ projecting to any $\A, \epsilon$-thin part, which are arbitrarily far away from all of $\T(S)$.

\begin{theorem}
\label{thm:Diameter_unbounded_precise}
Let $\A$ be any simple closed multicurve. For all $\epsilon > 0$ small enough, and for all $D > 0$, there exists a $\mu \in \PC_{fill}(S)$ so that 
\[
 \pi(\mu) \in \T_{\A,\epsilon}(S)
\]
and for all $X \in \T(S)$,
\[
 d(\mu, X) > D
\]
In other words, $\mu$ lies above the $\A, \epsilon$-thin part, and is distance at least $D$ from any point in Teichm\"uller space.
%
\end{theorem}
The proof will apply the lemmas in Section \ref{sec:When_curves_short} to the current from Lemma \ref{lem:Define_mu}. 
\begin{proof}
 Let $\epsilon, D > 0$. Let $N$ be the larger of the constant from Theorem \ref{thm:Short_curves} for the given $\epsilon$, and $e^{2D}$. Consider the current $\mu$ from Lemma \ref{lem:Define_mu} above, defined for the multicurve $\A$ and constant $N$. Let $\alpha \in \A$. By property (2) of Lemma  \ref{lem:Define_mu}, $i(\mu, \beta) > N i(\mu, \alpha)$ for any closed curve $\beta$ that intersects $\alpha$. So by Theorem  \ref{thm:Short_curves}, 
 \[
  \ell_{\pi(\mu)}(\alpha) < \epsilon
 \]
 for all $\alpha \in \A$.
 
 
 Now let $Y$ be a connected component of $S \setminus \A$. Each boundary curve of $Y$ has length at most $\epsilon$ in $\pi(\mu)$. So for all $\epsilon$ small enough, Theorem \ref{thm:Thick_part_guarantee} applies. Take the marking $\Gamma$ guaranteed by property (3) of Lemma  \ref{lem:Define_mu}. Then for any $\gamma \in \Gamma$, we have $i(\mu, \Gamma) \asymp 1$. But the current $\mu$ from Lemma  \ref{lem:Define_mu} is a sum of closed curves, so it has integer intersection with any simple closed curve. Thus, $\sys_Y(\mu) \geq 1$ and so, 
\[
 i(\mu, \gamma) \asymp \sys_Y(\mu)
\]
for any $\gamma \in \Gamma$. So by Theorem \ref{thm:Thick_part_guarantee}, for any essential non-peripheral simple closed curve $\beta$ on $Y$,
\[
 \ell_{\pi(\mu)}(\beta) \succ 1
\]
where the constant depends only on $S$. In other words, 
\[
\pi(\mu) \in \T_{\A, \epsilon}(S)
\]
for an implied constant $c$ depending only on $\chi(S)$.

Now we show that $\mu$ is at least distance $D$ away from any point $X \in \T(S)$. Fix a curve $\delta$ with one self-intersection, that lies in some connected component $Y$ of $S \setminus \A$. Next, choose $X$. By Buser \cite[Theorem 4.2.1]{Buser}, since $i(\delta, \delta) = 1$, we have $\ell_X(\gamma) \geq 1$. By property (4) of Lemma  \ref{lem:Define_mu}, $i(\mu, \delta) \asymp 1$. Thus,
\[
 \frac{\ell_X(\delta)}{i(\mu, \delta)} \succ 1
\]
Moreover, let $\gamma$ be a curve on $X$ shorter than the Bers constant $c_b$. As $\mu$ is a sum of closed curves, $i(\mu, \gamma) \geq 1$. In other words, 
\[
 \frac{i(\mu, \gamma)}{\ell_X(\gamma)} \succ 1
\]
where the constant depends only on $S$.
Therefore, 
\[
 \sup_\nu \frac{\ell_X(\nu)}{i(\mu, \nu)},\sup_\nu \frac{i(\mu, \nu)}{\ell_X(\nu)} \succ 1
\]
So to show that the distance from $\mu$ to $X$ is at least $D$, we just have to find one ratio of the form  $\frac{\ell_X(\nu)}{i(\mu, \nu)}$ or $\frac{i(\mu, \nu)}{\ell_X(\nu)}$ bounded below by $e^D$.

Take any $\alpha \in \A$, and let $\beta$ be the shortest curve that intersects $\alpha$. Then, from Lemma \ref{lem:Define_mu}, we have 
\[
 i(\mu, \alpha)\in \{1,2\}, \text{ and } i(\mu, \beta) > N 
\]
If $\ell_X(\beta) \leq \sqrt N$, then 
\[
 \frac{i(\mu, \beta)}{\ell_X(\beta)} \geq \sqrt N = e^D
\]
by definition of $N$.

So suppose $\ell_X(\beta) > \sqrt N$. As $\beta$ is the shortest curve that intersects $\alpha$, we have that 
\[
 \ell_X(\beta) \prec \Col_X(\alpha) + \ell(\alpha)
\]
for constants depending only on $\chi(S)$. This follows, for example, from \cite[Lemma 8.1]{Sapir_comparison}. Then either $\ell_X(\alpha) \geq \frac 12 \sqrt N$, or $\Col_X(\alpha) \geq \frac 12 \sqrt N$. In the first case, when $\ell_X(\alpha) \geq \frac 12 \sqrt N$, we have 
\[
 \frac{\ell_X(\alpha)}{i(\mu, \alpha)} \geq \frac 12 \sqrt N \geq \frac 12 e^D
\]
by definition of $N$. So suppose $\Col_X(\alpha) \geq \frac 12 \sqrt N$. When the collar width of a curve is large enough, we have the estimate $\Col_X(\alpha)\asymp - \log \ell_X(\alpha)$. So for $D$, and therefore $N$, large enough, $\ell_X(\alpha) \prec e^{-\sqrt N/2}$. In other words, 
\[
 \frac{i(\mu, \alpha)}{\ell_X(\alpha)} \succ e^{N/2} = e^{\frac 12 e^{D}}
\]

In each of these cases, we get $d(\mu, X) \succ D$, for a constant depending only on $S$. Clearly, we can multiply $N$ by a constant depending on $S$ if necessary to get $d(\mu, X) \geq D$ for all $X \in \T(S)$.

%
%
%
\end{proof}

 \bibliographystyle{alpha}
  \bibliography{ProjectionProperties2}
\end{document}